\documentclass[12pt]{amsart}
\usepackage{amsmath, amssymb, amsfonts, amsthm, mathtools}
\usepackage{hyperref}
\usepackage{tabularx}
\usepackage{booktabs}
\usepackage{caption}
\usepackage{aurical}
\usepackage{xcolor}

\newtheorem{thm}{Theorem}[section]
\newtheorem{lem}[thm]{Lemma}
\newtheorem{cor}[thm]{Corollary}
\newtheorem{prop}[thm]{Proposition}
\newtheorem{ex}[thm]{Example}

\newtheorem{qu}[thm]{Question}

\newtheorem{con}[thm]{Conjecture}

\newtheorem*{prob*}{Open problem}

\theoremstyle{definition}

\newtheorem{defi}[thm]{Definition}

\theoremstyle{remark}

\newtheorem{rem}[thm]{Remark}
\newtheorem*{rem*}{Remark}

\DeclareMathOperator{\Aff}{Aff}
\DeclareMathOperator{\Aut}{Aut}

\newcommand{\kringel}{\mathbin{\raise0.5pt\hbox{$\scriptstyle\circ$}}}
\newcommand{\pkt}{\mathbin{\raise0.5pt\hbox{$\scriptstyle\bullet$}}}
\newcommand{\sq}{\mathbin{\raise0.5pt\hbox{$\scriptscriptstyle\square$}}}

\newcommand{\ck}{\checkmark}

\newcommand{\A}{\mathbb{A}}
\newcommand{\C}{\mathbb{C}}
\newcommand{\HH}{\mathbb{H}}

\newcommand{\E}{\mathbb{E}}

\newcommand{\Q}{\mathbb{Q}}
\newcommand{\R}{\mathbb{R}}
\renewcommand{\S}{\mathbb{S}}
\newcommand{\Z}{\mathbb{Z}}

\newcommand{\tr}{\mathop{\rm tr}}
\newcommand{\ad}{{\rm ad}}

\newcommand{\End}{{\rm End}}
\newcommand{\Der}{{\rm Der}}
\newcommand{\Inn}{{\rm Inn}}
\newcommand{\Iso}{{\rm Iso}}

\newcommand{\La}{\mathfrak{a}}

\newcommand{\Lg}{\mathfrak{g}}

\newcommand{\Ll}{\mathfrak{l}}
\newcommand{\Ln}{\mathfrak{n}}

\newcommand{\Lr}{\mathfrak{r}}
\newcommand{\Ls}{\mathfrak{s}}

\newcommand{\CO}{\mathcal{O}}

\newcommand{\abs}[1]{\lvert#1\rvert}

\newcommand{\al}{\alpha}
\newcommand{\be}{\beta}
\newcommand{\ga}{\gamma}
\newcommand{\de}{\delta}
\newcommand{\ep}{\varepsilon}
\newcommand{\ka}{\kappa}

\newcommand{\om}{\omega}

\newcommand{\Ga}{\Gamma}

\newcommand{\ra}{\rightarrow}
\renewcommand{\phi}{\varphi}

\begin{document}


\title[Crystallographic actions]{Crystallographic actions on Lie groups and post-Lie algebra structures}


\author[D. Burde]{Dietrich Burde}
\address{Fakult\"at f\"ur Mathematik\\
Universit\"at Wien\\
  Oskar-Morgenstern-Platz 1\\
  1090 Wien \\
  Austria}
\email{dietrich.burde@univie.ac.at}

\date{\today}

\subjclass[2000]{Primary 20H15, 17D99, Secondary 22E40}
\keywords{Crystallographic groups, Pre-Lie algebras, Post-Lie algebras}

\begin{abstract}
This survey on crystallographic groups, geometric structures on Lie groups and associated algebraic structures
is based on a lecture given in the Ostrava research seminar in $2017$.
\end{abstract}

\maketitle

\tableofcontents

\section{Introduction}

Crystallographic groups and crystallographic actions already have a long history. 
They were studied more than hundred years ago as the symmetry groups of crystals 
in three-dimensional Euclidean space and as wallpaper groups in two-dimensional Euclidean space.
Such groups are discrete and cocompact subgroups of the group of isometries of a Euclidean space.
After Hilbert asked in $1900$ about Euclidean crystallographic groups in his $18$th problem, 
Bieberbach solved
this question in $1910$. Since then Euclidean crystallographic structures are quite well understood, and
several other types of crystallographic structures have been considered, such as almost-crystallographic and 
affine crystallographic structures. For the affine case it was expected that the results from the Euclidean case
should generalize in a straightforward manner. This, however, turned out to be not the case. The Bieberbach theorems
do not hold. In particular, groups admitting an affine crystallographic action need not be virtually abelian.
However, it was conjectured that such groups are virtually polycyclic. This became known as Auslander's conjecture.
J. Milnor proved several results on affine crystallographic actions in his fundamental paper \cite{MIL} in $1977$. 
See also the discussion in \cite{GOL} by W. M. Goldman. This resulted in an active research on 
affine crystallographic actions and more generally on nil-affine crystallographic actions until today.
We want to give a survey on these developments and state some results which we have obtained in this context.
An important step here is to be able to formulate the problems on the level of Lie algebras and in terms of pre-Lie
algebra and post-Lie algebra structures. \\
This survey is by no means complete and there are several other interesting 
results in this area, which we do not mention.

\section{Euclidean crystallographic actions}

Let $E(n)$ denote the isometry group of the Euclidean space $\R^n$. This group is given 
by matrices as follows,
\[
E(n)=\left\{
\begin{pmatrix} A & v \\ 0 & 1 \end{pmatrix} \mid A\in O_n(\R), v\in \R^n\right\}.
\]

The multiplication is the usual matrix multiplication
\[
\begin{pmatrix} A & v \\ 0 & 1 \end{pmatrix}
\begin{pmatrix} B & w \\ 0 & 1 \end{pmatrix}
=\begin{pmatrix} AB & Aw+v \\ 0 & 1 \end{pmatrix}.
\]
The translations form a normal subgroup of $E(n)$, given by
\[
T(n)=\left\{
\begin{pmatrix} E_n & v \\ 0 & 1 \end{pmatrix} \mid  v\in \R^n\right\}.
\]
In particular we have  $E(n)\cong O_n(\R) \ltimes T(n) \cong  O_n(\R)\ltimes \R^n$.
The group $E(n)$ acts on $\R^n$ by
\[
\begin{pmatrix} A & b \\ 0 & 1 \end{pmatrix} \begin{pmatrix} v \\ 1 \end{pmatrix}=\begin{pmatrix} Av+ b \\ 1 \end{pmatrix}.
\]
More generally, the {\em affine group of the Euclidean space $\R^n$}, denoted by $A(n)$, is given as follows
\[
A(n)=\left\{
\begin{pmatrix} A & v \\ 0 & 1 \end{pmatrix} \mid A\in GL_n(\R), v\in \R^n\right\}.
\]
We will use the following definition of an Euclidean crystallographic group (ECG), which can be found 
in \cite{AUS2}, section $1$.

\begin{defi}\label{2.1}
A {\em Euclidean crystallographic group} $\Ga$ is a subgroup of $E(n)$ which is discrete and cocompact, i.e.,
has compact quotient $\R^n/\Ga$.
\end{defi}

Let $\Ga$ be an {\em ECG}. Then $\Ga$ acts {\em properly discontinuously} on $\R^n$, i.e., for all compact sets 
$K\subseteq \R^n$ the set of returns
\[
\{ \ga \in \Ga \mid \ga K\cap K\neq \emptyset \}
\]
is finite. 

\begin{defi}
Let $\Ga$ be a group acting on $\R^n$ via a homomorphism
\[
\rho\colon \Gamma \rightarrow E(n).
\]
The action is called a {\em crystallographic action}, if $\Gamma$ acts properly discontinuously on $\R^n$
and the orbit space $\R^n/\Gamma$ is compact, i.e., $\Ga$ acts cocompactly.
\end{defi}

\begin{ex}
The group $\Ga=\Z^n$ acts crystallographically on $\R^n$ by translations.
\end{ex} 

A homomorphism $\rho\colon \Gamma \rightarrow E(n)$ determines a crystallographic action if and only if the kernel of 
$\rho$ is finite, and the image of $\rho$ is a crystallographic group. \\[0.2cm]
As already said, the study of ECGs has a long history. ECGs in dimension $2$ are the $17$ wallpaper groups,
which have been known for several centuries. However the proof that the list was complete was only given in 
$1891$ by Fedorov, after the much more difficult classification in dimension $3$ had been completed 
by Fedorov and independently by Sch\"onflies in $1891$. There are $219$ distinct ECGs in dimension $3$.
All of them are realized as the symmetry group of genuine crystals.
Hilbert published in $1900$ his famous $23$ problems \cite{HIL}. In the first part of the $18$th problem
he asked, whether there are only finitely many different crystallographic groups in any dimension. This was answered 
affirmatively by Bieberbach \cite{BI1,BI2} in $1910$. His theorems are usually stated as follows.

\begin{prop}[Bieberbach 1]
Let $\Gamma \le E(n)$ be an Euclidean crystallographic group. Then $\Gamma$ contains the translation subgroup
$\Z^n$ as a normal subgroup with finite quotient $F=\Ga/\Z^n$.
\end{prop}

\begin{prop}[Bieberbach 2]
Two Euclidean crystallographic groups in dimension $n$ are isomorphic if and only if they are conjugated in the affine group
$A(n)\cong GL_n(\R)\ltimes \R^n$.
\end{prop}

\begin{prop}[Bieberbach 3]
In each dimension there are only finitely many Euclidean crystallographic groups.
\end{prop}

Let $\Ga \le E(n)$ be an ECG. By the first Bieberbach Theorem, the translation subgroup is an abelian subgroup of finite index, 
isomorphic to the full lattice $\Z^n$ in $\R^n$. Hence we have a short exact sequence
\[
1\ra \Z^n\ra \Ga \ra F\ra 1
\]
with a finite group $F\cong\Gamma/\Z^n$ acting by conjugation of $\Z^n$. We obtain a faithful representation
\[
F\hookrightarrow \Aut(\Z^n)=GL_n(\Z),
\]
so that we may consider $F$ as a finite subgroup in $GL_n(\Z)$ up to conjugation. Now $GL_n(\Z)$ has only finitely many 
conjugacy classes of finite subgroups. This was first shown by Jordan, then by Zassenhaus and more generally later by 
Harish-Chandra for arithmetic groups. Hence we have only finitely many such groups $F$ up to isomorphism with given 
action of $F$ on $\Z^n$. Furthermore there are only finitely many inequivalent extensions $1\ra \Z^n\ra \Ga \ra F\ra 1$, 
since the extension classes are classified by the group
\[
H^2(F,\Z^n)\cong H^1(F,\Q^n/\Z^n),
\]
which is discrete and compact, hence finite. This yields Bieberbach's third Theorem. \\[0.2cm]
Zassenhaus \cite{ZAS} showed that every ECG arises as an exact sequence as above 
and gave an algorithm yielding the $n$-dimensional ECGs up to affine equivalence, given the finitely many 
conjugacy classes of finite subgroups $F$ in $GL_n(\Z)$ together with their normalizers.

\begin{prop}
For a given finite group $F \le GL_n(\Z)$ the isomorphism classes of crystallographic groups $\Gamma$ with
conjugacy class represented by $F$ are in bijection with the orbits of the normalizer $N_{GL_n(\Z)}(F)$ on the
finite group $H^2(F,\Z^n)$. 
\end{prop}

In $1978$ the classification in dimension $4$ was achieved in \cite{BBNWZ}.

\begin{prop}[Zassenhaus et al. 1978]
There are exactly $4783$ different crystallographic groups in four-dimensional space $\R^4$.
\end{prop}

In $2000$ Plesken and Schulz \cite{PLS} classified all ECGs in dimension $5$ and $6$. There are $222018$ different
ECGs in dimension $5$ and $28927922$ different ECGs in dimension $6$.

\begin{ex}
The group $GL_2(\Z)$ has exactly $13$ different conjugacy classes of finite subgroups, called arithmetic ornament classes.
Zassenhaus' algorithm yields $17$ ECGs up to isomorphism.
\end{ex}

It is easy to see that the $13$ arithmetic ornament classes are given as follows:
\begin{align*}
C_1 & \cong \left\langle \begin{pmatrix} 1 & 0 \\ 0 & 1 \end{pmatrix} \right\rangle,\; 
C_2 \cong \left\langle \begin{pmatrix} -1 & 0 \\ 0 & -1 \end{pmatrix} \right\rangle,\;
C_3 \cong \left\langle \begin{pmatrix} 0 & 1 \\ -1 & -1 \end{pmatrix} \right\rangle, \;
C_4  \cong \left\langle \begin{pmatrix} 0 & 1 \\ -1 & 0 \end{pmatrix} \right\rangle, \; \\
C_6 & \cong \left\langle \begin{pmatrix} 0 & 1 \\ -1 & 1 \end{pmatrix} \right\rangle, \;
D_1 \cong \left\langle \begin{pmatrix} 1 & 0 \\ 0 & -1 \end{pmatrix} \right\rangle, \;
D_1 \cong \left\langle \begin{pmatrix} 0 & 1 \\ 1 & 0 \end{pmatrix} \right\rangle, \;
D_2 \cong  \left\langle \begin{pmatrix} 1 & 0 \\ 0 & -1 \end{pmatrix},\begin{pmatrix} -1 & 0 \\ 0 & -1 \end{pmatrix}
\right\rangle, \\
D_2 & \cong  \left\langle \begin{pmatrix} 0 & 1 \\ 1 & 0 \end{pmatrix},\begin{pmatrix} -1 & 0 \\ 0 & -1 \end{pmatrix}
\right\rangle, \; D_3 \cong \left\langle \begin{pmatrix} 0 & 1 \\ 1 & 0 \end{pmatrix},
\begin{pmatrix} 0 & 1 \\ -1 & -1 \end{pmatrix}\right\rangle, \;
D_3  \cong  \left\langle \begin{pmatrix} 0 & -1 \\ -1 & 0 \end{pmatrix},\begin{pmatrix} 0 & 1 \\ -1 & -1 \end{pmatrix}
\right\rangle, \\
D_4 & \cong  \left\langle \begin{pmatrix} 0 & 1 \\ 1 & 0 \end{pmatrix},\begin{pmatrix} 0 & 1 \\ -1 & 0 \end{pmatrix}
\right\rangle, \; D_6\cong  \left\langle \begin{pmatrix} 0 & 1 \\ 1 & 0 \end{pmatrix},
\begin{pmatrix} 0 & 1 \\ -1 & 1 \end{pmatrix}\right\rangle. 
\end{align*}
The groups are isomorphic to one of the cyclic groups $C_1,C_2,C_3,C_4,C_6$, or one of the dihedral groups 
$D_1,D_2,D_3,D_4,D_6$. The conjugacy classes are finer than the isomorphism classes.
The indices $1,2,3,4,6$ are no coincidence here. An element $A\in GL_2(\Z)$ of finite order has one of the orders
$1,2,3,4,6$. Indeed, if there is an element $A\in GL_2(\Z)$ of order $n$, then $\phi(n)$, the degree of
the irreducible cyclotomic polynomial $\Phi_n$ divides $2$ by Cayley-Hamilton. But $\phi(n)\mid 2$ is equivalent
to $n=1,2,3,4,6$. The wallpaper groups $\Ga$ arise from these $13$ arithmetic ornament classes by equivalence classes of 
extensions $1\ra \Z^2\ra \Ga\ra F\ra 1$, determined by $H^2(F,\Z^2)$. For each of these classes we can compute this group.
In case that $H^2(F,\Z^2)=0$ the ornament class just yields one extension. This happens in $10$ cases. In the other
three cases $H^2(F,\Z^2)$ is isomorphic to $\Z/2$, $\Z/2$ and $\Z/2\times \Z/2$ respectively. This yields $2+2+4=8$
further extensions. Altogether we obtain $18$ inequivalent extensions leading to $17$ different groups. 

\begin{rem}
Denote by $c(n)$ the number of different ECGs in $E(n)$. Peter Buser \cite{BUS} showed in $1985$, using Gromov's work on
almost flat manifolds, the estimate
\[
c(n)\le e^{e^{4n^2}}.
\]
This bound seems to be not yet optimal, but I haven't found better estimates. Schwarzenberger \cite{SCH} has shown that 
$c(n)$ grows at least as fast as $2^{n^2}$ and conjectured that this is the exact asymptotic result. This seems to be
still open.
\end{rem}

By Bieberbach's first Theorem the translation group of an ECG  is an abelian subgroup of finite index. Hence
every ECG is virtually abelian. We can reformulate the structure results by Bieberbach as follows.

\begin{prop}\label{2.11}
The groups admitting a Euclidean crystallographic action are precisely the finitely generated virtually abelian groups. 
For a given group the crystallographic action is unique up to affine conjugation.
\end{prop}

\begin{defi}
An Euclidean crystallographic group $\Ga\le E(n)$ is called a {\em Bieberbach group}, if it is torsionfree, i.e., if it acts 
freely on $\R^n$.
\end{defi}

If $M = \R^n/\Ga$ is a compact complete connected flat Riemannian manifold, then its fundamental group $\pi_1(M)\cong \Ga$ 
is a Bieberbach group. Conversely, every flat complete Riemannian manifold $M$ is 
the quotient $\R^n/\Ga$ for a subgroup $\Ga\le E(n)$ acting freely and properly discontinuously on $\R^n$. This shows the 
geometric importance of Bieberbach groups. Among the $17$ wallpaper groups, there are just $2$ Bieberbach groups, namely the 
fundamental group of the torus and the Klein bottle. Among the $219$ space groups there are only $10$ Bieberbach groups.
In dimension $4,5,6$ we have $74,1060,38746$ Bieberbach groups, so also these numbers grow rapidly.

\section{Hyperbolic and spherical  crystallographic actions}

ECGs have been generalized to non-Euclidean crystallographic groups, namely to spherical and hyperbolic
crystallographic groups. We will shortly explain the notions and give a few examples, but we will not attempt to
give a survey. Let $X$ be a {\em space of constant curvature} $X$, i.e., a simply-connected complete
Riemannian manifold of constant curvature $\ka$ up to scaling, together with its isometry group $G=\Iso(X)$. 
Any space of constant curvature is isomorphic to either the Euclidean space $(\E^n,O_n(\R)\ltimes \R^n)$ with $\ka=0$, or 
to the sphere $(\S^n,O^+(n+1))$ with $\ka=1$, or to the hyperbolic space $(\HH^n,O^+(n+1))$ with $\ka=-1$. Here $O^+(n,1)$
is the index $2$ subgroup of $O(n,1)$ preserving the two connected components of $\{A \in \R^{n+1}\mid \langle 
A,A\rangle=-1\}$. Then Definition $\ref{2.1}$ is generalized as follows.

\begin{defi}
Let $(X,G)$ be a space of constant curvature. A subgroup $\Ga\le G=\Iso(X)$ is called a {\em crystallographic group}, or CG,
if $\Ga$ is discrete and $X/\Ga$ has finite volume.
\end{defi}

Any discrete subgroup $\Ga\le \Iso(X)$ has a convex fundamental domain. So for an ECG any fundamental domain is bounded
since any unbounded convex domain in Euclidean space has infinite volume. Hence any CG in $\E^n$ is cocompact and thus
an ECG. This shows that both definitions coincide for ECGs. Any CG in $\S^n$ is a discrete subgroup of a compact group
$O_n(\R)$ and hence finite. So spherical CRs are finite subgroups of $O_n(\R)$. For small $n$, all finite subgroups of
$O_n(\R)$ are classified. For example, any finite subgroup of $O_2(\R)$ is either cyclic or dihedral. For higher $n$ this is
not the case. A special case of the Margulis lemma implies that for each $n$, there is a positive integer $m(n)$ such that any 
finite subgroup of $O_n(\R)$ has an abelian subgroup of index $m(n)$, see Corollary 4.2.4 of Thurston's book \cite{THU}.
The most interesting case of non-Euclidean CGs is the hyperbolic case. Already in dimension $2$ there is a continuum of CGs,
even of cocompact ones. The latter arise as fundamental groups of closed surfaces of genus $g>1$. their totality can be 
described via Teichm\"uller theory.

\begin{ex}
Let $X=\HH^2$ be the upper half-plane, $G=\Iso(X)\cong {\rm PSL}_2(\R)$ and $\Ga$ be the modular group consisting
of transformations of the form
\[
z\mapsto \frac{az+b}{cz+d}, \quad \begin{pmatrix} a & b \\ c & d \end{pmatrix}
\in SL_2(\Z)
\]
Then $\Ga$ is a non-cocompact hyperbolic CG with ${\rm vol} (\HH^2/\Ga)=\frac{\pi}{3}$.
\end{ex}

Another example is given by Bianchi groups.
\begin{ex}
Let $d$ be a positive squarefree integer and $K=\Q(\sqrt{-d})$ an imaginary-quadratic number field. Denote by $\CO_d$ 
its ring of integers in $K$. Let $\Ga(d)={\rm PSL}_2(\CO_d)\subset {\rm PSL}_2(\C)$. Then $\Ga(d)$ is a discrete subgroup 
of $\Iso(\HH^3)$, called a Bianchi group. It is a non-compact hyperbolic CG.
\end{ex}
In fact, the covolume of $\Ga(d)$ is given by
\[
{\rm vol}(\HH^3/\Ga(d))=\frac{\abs{d_K}^{3/2}}{4\pi^2}\zeta_K(2),
\]
where $d_K$ denotes the discriminant of $K$ and $\zeta_K(s)$ denotes the Dedekind zeta function of the base field 
$K= \mathbb{Q}(\sqrt{-d})$. \\[0.2cm]
There is a general method of constructing arithmetic discrete subgroups of semisimple Lie groups due to Margulis.
On the other hand, there exist also non-arithmetic hyperbolic CGs in any dimension \cite{GRP}. By {\em Mostow rigidity},
any isomorphism of hyperbolic CGs in $\HH^n$ for $n\ge 3$ is induced by a conjugation in the group $\Iso(\HH^n)$.
This is far from being true for $n=2$, see above. An important numerical invariant of a hyperbolic CG is its covolume
\[
v(\Ga)={\rm vol}(\HH^n/\Ga).
\]
For $n\ge 4$ the set of covolumes is discrete and for $n=3$ it is a non-discrete closed well-ordered set of order-type 
$\om^{\om}$, where each point has finite multiplicity. The covolumes are bounded from below by a positive
constant depending only on $n$. There is much more to say, but we will finish this section by referring to related topics
such as {\em lattices in Lie groups}, {\em Fuchsian groups} and {\em Kleinian groups}. A {\em Fuchsian group} is a discrete 
subgroup of ${\rm PSL}_2(\R)$ and a {\em Kleinian group} is a discrete subgroup of ${\rm PSL}_2(\C)$.

\section{Affine and nil-affine crystallographic actions}

We are mainly interested in this survey in another generalization of Euclidean crystallographic groups, namely in affine and
nil-affine crystallographic groups. Let $X$ be a locally compact topological Hausdorff space and $G$ the group of 
homeomorphisms of $X$.

\begin{defi}
A subgroup $\Ga$ of $G$ is called {\em crystallographic}, if $G$ acts properly discontinuously and cocompactly on $X$.
A continuous action of a group $\Gamma$ on $X$ is called a {\em crystallographic action}, if it is properly discontinuous and 
cocompact.
\end{defi}

For $X$ being the affine space $\A^n$ and $G=A(n)$, a crystallographic group $\Ga$ is called an ACG, an {\em affine 
crystallographic group}. For $X=\E^n$, a group $\Ga \le G=E(n)$ acts properly discontinuously on $X$ if and only if
$\Ga$ is discrete. In general acting properly discontinuously is stronger than being discrete. The group $A(n)$ is
a generalization of the Euclidean isometry group $E(n)$ as we have seen in section $2$ in the context of the Bieberbach
theorems. As in the Euclidean case, torisonfree ACGs arise as fundamental groups of flat manifolds, i.e., of complete
compact affinely flat manifolds. A natural question in this context is whether the Bieberbach theorems hold for ACGs.
Looking at some examples it is clear that this is not the case.

\begin{ex}
Let $k$ be a fixed integer. The group
\[
\Ga_k= \Bigg\{\begin{pmatrix} 
1 & kc & 0 & \vrule & ka \\
0 & 1  & 0 & \vrule & kb \\
0 & 0  & 1 & \vrule & kc \\
\hline \\[-0.48cm]
0 & 0  & 0 & \vrule & 1 
\end{pmatrix}
\mid a,b,c\in \Z \Bigg\}
\le A(3)
\]
is a $3$-dimensional ACG, which is not virtually abelian. Because of
\[
\Ga_k/[\Ga_k,\Ga_k]\cong \Z^2\oplus \Z/k\Z 
\]
two groups $\Ga_k$ and $\Ga_{k'}$ are isomorphic if and only if $k=k'$. Hence there are infinitely many different
ACGs in dimension $3$.
\end{ex}

Indeed, every abelian subgroup of finite index in $\Ga_k$ would necessarily be isomorphic to $\Z^3$, but is is easy to see 
that $\Ga_k$ does not contain such a subgroup. So ACGs need not be virtually abelian. On the other hand, all such
examples are virtually solvable and then, being a discrete solvable subgroup of a Lie group with finitely many components,
already virtually {\em polycyclic}. Is this a possible generalization of Bieberbach's First Theorem, i.e., is it
true that every ACG is virtually polycyclic? In other words, is the fundamental group of every complete compact affine 
manifold virtually polycylic? 
L. Auslander studied this problem and published a paper \cite{AUS2} in $1964$ stating an even more general result, namely
that the fundamental group of every complete affine manifold is virtually polycyclic, without the compactness
assumption. Unfortunately his proof was in error. Nevertheless the statement later on became widely known as 
{\rm the Auslander conjecture}:

\begin{con}[Auslander]
Every ACG is virtually polycyclic.
\end{con}

The history of this conjecture is as follows. In $1977$ J. Milnor studied the fundamental groups of flat affine manifolds in 
his famous paper \cite{MIL}. He proved that every torsion-free virtually polycyclic group can be realized as the fundamental 
group of some complete flat affine manifold. Then he conjectured also the converse, namely Auslander's statement that the 
fundamental group of every complete flat affine manifold is virtually polycyclic. However, Margulis \cite{MAR} found a 
counterexample in dimension $3$. 

\begin{prop}[Margulis]
There exists non-compact complete affine manifolds in dimension $3$ with a free non-abelian fundamental group
of rank $2$.
\end{prop}

A free non-abelian group cannot be virtually polycyclic. So it is clear that one needs the compactness assumption and 
Auslander's original claim cannot hold. This led to the formulation of the Auslander conjecture in terms of of affine 
crystallographic groups. \\
Auslander's conjecture is still open, although many special cases are known. Fried and Goldman proved the conjecture in 
$1983$ in dimension $n\le 3$. Tomanov \cite{TOM} proved it in $2016$ for $n\le 5$ and Abels, Margulis and Soifer \cite{AMS}
have worked on several cases for many years. In $2005$ they showed that every crystallographic subgroup
$\Ga\le A(n+2)$ with linear part contained in $O(n,2)$ is virtually polycyclic \cite{AMS2}.
They have a proof for Auslander's conjecture in dimension $6$ available on the arXiv. However, they have withdrawn it now. \\[0.2cm]
In his paper \cite{MIL} Milnor also asked the following important question.

\begin{qu}[Milnor 1977]\label{4.5}
Does every virtually polycyclic group admit an affine crystallographic action?
\end{qu}

Actually, the original question uses the terminology of left-invariant affine structures on Lie groups, see section $6$.
Because of some positive evidence this question was also sometimes called the {\em Milnor conjecture}.
A positive answer for both Milnor and Auslander would give a very nice algebraic description 
of the class of groups admitting an affine crystallographic action. In fact, then this class would be precisely the class
of virtually polycyclic groups and we would have a perfect analogue to Proposition $\ref{2.11}$ concerning the Euclidean case.
Moreover, it is known that an affine crystallographic action of a virtually polycyclic group is unique up to conjugation 
with a polynomial diffeomorphism of $\R^n$. \\
However, Y. Benoist found a counterexample to Milnor's conjecture in \cite{BEN} and we provided families of counterexamples 
in \cite{BU3,BU5}. The counterexamples in \cite{BU5} are torsion-free nilpotent groups of Hirsch length $11$ and 
nilpotency class $10$ not admitting an affine crystallographic action.
Hence one needs to replace $A(n)$ by a larger group for such a correspondence to hold.
Indeed, other alternatives have been proposed. First it was shown that the group $P(n)$ of polynomial diffeomorphisms
of $\R^n$ is a possible alternative. In \cite{DEI} it was shown that any virtually polycyclic group admits a polynomial 
crystallographic action of bounded degree. However, this group appears to be too large and does not have such a geometric
meaning as $E(n)$ and $A(n)$ have. A more natural generalization of $A(n)=\Aff(\R^n)$ turned out to be
the group $\Aff(N)=\Aut(N)\ltimes N$, the group of {\em nil-affine transformations}, in this context. 
Here $N$ denotes a connected and simply-connected nilpotent Lie group. For the abelian Lie group $N=\R^n$ we recover the
group $A(n)$. We repeat the definition of a crystallographic action for $\Aut(N)$.

\begin{defi}
A {\em nil-affine crystallographic action} consists of a representation $\rho\colon \Ga \ra \Aff(N)$
for some connected and simply connected nilpotent Lie group $N$ letting $\rho$ act properly discontinuously and
cocompactly on $N$. The image $\rho(\Ga)$ of such an nil-affine crystallographic action will be referred to as an 
{\em nil-affine crystallographic group}.
\end{defi}

In $2003$ K. Dekimpe showed the following result in \cite{DEK}.

\begin{prop}
Every  virtually polycyclic group $\Ga$ admits a nil-affine crystallographic action $\rho\colon \Ga\ra \Aff(N)$. 
This action is unique up to conjugation inside of $\Aff(N)$.
\end{prop}

It is now also natural to ask for the converse, i.e., to ask for the Auslander conjecture for nil-affine crystallographic groups.

\begin{con}[Generalized Auslander]
Let $N$ be a connected and simply connected nilpotent Lie group and let
$\Gamma \subseteq \Aff(N)$ be a group acting crystallographically on $N$.
Then $\Gamma$ is virtually polycyclic.
\end{con}

If this conjecture has a positive answer, then we have an analogue of Proposition $\ref{2.11}$ for nil-affine 
crystallographic actions. Then the groups admitting a nil-affine crystallographic action would be precisely the
virtually polycyclic groups. \\[0.2cm]
We have shown in \cite{BU19} that the generalized Auslander conjecture is true for $n\le 5$ and that
it can be reduced to the ordinary Auslander conjecture in case $N$ is $2$-step nilpotent.

\section{Simply transitive groups of affine and nil-affine transformations}

Affine and nil-affine crystallographic actions of discrete groups are closely related to simply transitive actions
by affine and nil-affine transformations of Lie groups.

\begin{defi}
A group $G$ acts {\em simply transitively} on $\R^n$ by affine transformations if there is a homomorphism 
$\rho\colon G \ra A(n)$ letting $G$ act on $\R^n$, such that for all $y_1,y_2\in \R^n$ there is a unique 
$g\in G$ such that $\rho(g)(y_1)=y_2$.
\end{defi}

Such groups are connected and simply connected $n$-dimensional Lie groups which are homeomorphic to $\R^n$.
L. Auslander named such groups {\em simply transitive groups of affine motions}. He proved that such groups
are solvable \cite{AUS3}. We mention the following generalization of this result.

\begin{prop}
Let $G$ be a Lie group which is homeomorphic to $\R^n$ for some $n\ge 1$. If $G$ admits a faithful linear representation 
then $G$ is solvable.
\end{prop}

\begin{proof}
Let $G$ be a connected Lie group. By a theorem of Malcev and Iwasawa, $G$ is homeomorphic
to $C\times \R^k$ for some $k$, where $C$ is the maximal compact subgroup of $G$. If we assume that
$G$ is homeomorphic to $\R^n$ then it follows that $G$ has no nontrivial compact subgroup. \\
Since $G$ has a faithful linear representation, it is the semidirect product $B\ltimes H$ with a reductive group $H$ 
and a simply connected solvable group $B$, which is normal in $G$. This reduces the proof to the case where $G$ is reductive.
We have to show that our group is trivial then. \\
A reductive group $G$ having a faithful linear representation has a compact center $Z$
with semisimple quotient $G/Z$. So we may assume that $G$ is semisimple and has trivial
center. \\
A semisimple group $G$ with trivial center is analytically isomorphic to its adjoint group
and hence has a non-trivial compact subgroup unless $G$ is trivial. But since our $G$ has no
nontrivial compact subgroup it is trivial.
\end{proof}

Let us explain the connection between crystallographic actions of a discrete group and simply transitive actions of 
a Lie group. If $G$ is a solvable Lie group admitting a simply transitively action by affine transformations on $\R^n$, 
then a {\em cocompact lattice $\Gamma$} in $G$ admits an affine crystallographic action. Conversely, if a torsionfree 
nilpotent group $\Gamma$ admits an affine crystallographic action via $\rho\colon \Gamma \ra A(n)$, then $\rho(\Gamma)$ 
is unipotent.
Hence its {\em Malcev completion $G_{\Gamma}$} is inside $A(n)$, and acts simply transitively by affine transformations.
We have the following result.

\begin{prop}\label{5.3}
There is a bijective correspondence between affine crystallographic actions of a finitely-generated
torsionfree nilpotent group $\Ga$ and simply transitive actions by affine transformations of
its Malcev completion $G_{\Gamma}$.
\end{prop}

This generalizes to nil-affine crystallographic actions. We say that $G$ admits a simply transitively action by
{\em nil-affine transformations} on $N$, if there is a homomorphism $\rho\colon G \ra \Aff(N)$ letting $G$ act simply 
transitively on $N$.

\section{Left-invariant affine structures on Lie groups}

Milnor formulated his question $\ref{4.5}$ in terms of left-invariant affine structures on Lie groups.
We follow here his article \cite{MIL}.

\begin{defi}
An {\em affine} structure (or affinely flat structure) on an $n$-dimensional manifold $M$
is a collection of coordinate homeomorphisms
\[
f_{\al}\colon U_{\al}\ra V_{\al}\subseteq \R^n,
\]
where the $U_{\al}$ are open sets covering $M$, and the $V_{\al}$ are open
subsets of $\R^n$; whenever $U_{\al}\cap U_{\be}\neq \emptyset$, it is required that
the change of coordinate homeomorphism
\[
f_{\be}f_{\al}^{-1} \colon f_{\al}(U_{\al}\cap U_{\be})\ra f_{\be}(U_{\al}\cap U_{\be})
\]
extends to an affine transformation in $A(n)=\Aff(\R^n)$. We call $M$ together with this structure
an {\em affine} manifold, or an {\em affinely flat manifold}.
\end{defi}

A special case of affinely flat manifolds are {\em Riemannian-flat manifolds}, where
the coordinate changes extend to isometries in $E(n)$, i.e., to affine transformations
$x\mapsto Ax+b$ with $A\in O_n(\R)$. \\
For surfaces we have the following result by Benzecri \cite{BEZ}.

\begin{prop}
A closed surface admits an affine structure if and only if its Euler characteristic vanishes. 
\end{prop}

In particular, a closed surface different from the $2$-torus or the Klein bottle does not admit any affine 
structure.

\begin{defi}
An affine structure on a Lie group $G$ is called {\em left-invariant} if each left-multiplication map 
$L(g)\colon G\ra G$ is an affine diffeomorphism.
\end{defi}

\begin{defi}
An affine structure on $G$ is called {\em complete}, if the universal covering $\widetilde{G}$ is affinely 
diffeomorphic to $\R^n$.
\end{defi}

\begin{prop}
There is a canonical bijection between left-invariant complete affine
structures on $G$ and simply transitive actions of $G$ on $\R^n$ by affine motions.
\end{prop}

If $G$ admits a left-invariant complete affine structure, then for any discrete group
$\Ga$ the coset space $G/\Ga$ is a complete affinely flat manifold with fundamental group isomorphic to $\Ga$. \\[0.2cm]
Here is Milnor's question in the original context.

\begin{qu}[Milnor 1977]\label{6.6}
Does every solvable $n$-dimensional Lie group $G$ admit a complete left-invariant affine 
structure, or equivalently, does the universal covering group $\widetilde{G}$ act simply transitively
by affine transformations on $\R^n$ ?
\end{qu}

Milnor remarked that the answer is positive for $2$-step nilpotent and $3$-step nilpotent Lie groups and
for Lie groups whose Lie algebra admits a non-singular derivation. Such Lie algebras are necessarily nilpotent.
Furthermore the answer is positive for all connected and simply connected complex nilpotent Lie groups of
dimension $n\le 7$. However, as we have mentioned above, Benoist gave a counterexample in \cite{BEN} and we gave 
families of counterexamples in \cite{BU3,BU5}. The Lie algebras of all such counterexamples here are {\em filiform
nilpotent} Lie algebras. One can verify that all connected and simply-connected filiform nilpotent Lie groups
of dimension $n\le 9$ admit a complete left-invariant affine structure. Hence the minimal dimension for this kind of 
counterexamples is $10$. The result in \cite{BU5} is the following.

\begin{prop}
There exist families of nilpotent Lie groups of dimension $10$ and
nilpotency class $9$ not admitting any left-invariant affine structure.
\end{prop}

There are also families of such counterexamples in dimension $11,12$ and $13$, but a general result for
all dimensions $n\ge 10$ is only conjectured, see \cite{BU35}, but not known.

\section{Pre-Lie algebra and post-Lie algebra structures}

Several statements from the previous sections can be formulated on the level of Lie algebras in terms of
certain compatible algebraic structures on the Lie algebra of the corresponding Lie group. In particular, Milnor's 
question can be reduced to the level of Lie algebras, namely to pre-Lie algebra structures on Lie algebras and 
to faithful finite-dimensional representations of Lie algebras. All bijective correspondence mentioned are understood up to 
suitable equivalence of the structures involved.

\begin{defi}
A {\em pre-Lie algebra $(V,\cdot)$} is a vector space $V$ equipped with a binary operation $(x,y)\mapsto x\cdot y$ such 
that for all $x,y,z \in V$
\begin{align*}
(x\cdot y)\cdot z -x\cdot (y\cdot z) & = (y\cdot x)\cdot z - y\cdot (x\cdot z).
\end{align*}
\end{defi}

If $(V,\cdot)$ is a pre-Lie algebra, then for $x,y\in V$ the binary operation
\[
[x,y]:=x\cdot y-y\cdot x
\]
defines a Lie algebra.

\begin{defi}
A bilinear product $x\cdot y$ on $\Lg\times \Lg$ is called a {\em pre-Lie algebra structure
on $\Lg$}, if it satisfies
\begin{align*}
[x,y] & = x\cdot y-y\cdot x, \\
[x,y]\cdot z & = x\cdot (y\cdot z)-y\cdot (x\cdot z),
\end{align*}
for all $x,y,z\in \Lg$.
A Lie algebra $\Lg$ over a field $K$ is said to {\em admit a pre-Lie algebra structure}, if there
exists a pre-Lie algebra structure on $\Lg$.
\end{defi}

\begin{ex}
The Heisenberg Lie algebra $\Ln_3(K)$ of dimension $3$ with basis $\{e_1,e_2,e_3\}$ and
Lie brackets $[e_1,e_2]=e_3$ admits a pre-Lie algebra structure, given by
\[
e_1\cdot e_2  = \frac{1}{2}e_3,\quad e_2\cdot e_1  = -\frac{1}{2}e_3.
\]
\end{ex}

Denote by $L(x)$ the left multiplication operator given by
$L(x)(y)=x\cdot y$.  Then the second identity becomes
\[
L([x,y])= [L(x),L(y)]. 
\]
for all $x,y\in \Lg$. Hence the left multiplication operators define a $\Lg$-module $\Lg_L$ by
\[
L\colon \Lg \ra \Lg\Ll (\Lg),  x \mapsto L(x).
\]
Denote by $I\colon \Lg\ra \Lg_L$ the identity map. Then the first identity becomes
\[
I([x,y]) = I(x)\cdot y-I(y)\cdot x. 
\]
Hence the identity map is a $1$-cocycle, i.e., $I\in Z^1(\Lg,\Lg_L)$. We have the following result \cite{BU2}.

\begin{prop}
Let $\Lg$ be a $n$-dimensional Lie algebra. Then $\Lg$ admits a pre-Lie algebra structure if and only if
there is a $n$-dimensional $\Lg$-module $M$ with nonsingular $1$-cocycle in $Z^1(\Lg,M)$.
\end{prop}

\begin{ex}
Let $\Lg$ be a Lie algebra admitting a nonsingular derivation $D$. Then $\Lg$ admits a pre-Lie algebra
structure, given by
\[
x\cdot y=D^{-1}([x,D(y)])
\]
for all $x,y\in \Lg$.
\end{ex}

Jacobson \cite{JAC} proved the following result in $1955$.

\begin{prop}
Let $\Lg$ be a Lie algebra over a field of characteristic zero admitting a nonsingular derivation.
Then $\Lg$ is nilpotent.
\end{prop}

This result does not hold for fields of prime characteristic $p>0$. There are even
simple modular Lie algebras of nonclassical type admitting nonsingular derivations, see \cite{BKK}.
This is of interest in the theory of pro-$p$ groups of finite coclass. \\[0.2cm]
In general a given Lie algebra need not admit a pre-Lie algebra structure.

\begin{ex}
The Lie algebra $\Ls\Ll_2(K)$ over a field $K$ of characteristic zero does not
admit a pre-Lie algebra structure.
\end{ex}

More generally, we have the following result. 

\begin{prop}\label{7.8}
Let $\Lg$ be a finite-dimensional semisimple Lie algebra over a field of characteristic zero.
Then $\Lg$ does not admit a pre-Lie algebra structure.
\end{prop}

\begin{proof}
Let $\Lg$ be $n$-dimensional. Suppose that $\Lg$ admits a pre-Lie algebra structure. Then we have 
$I\in Z^1(\Lg,\Lg_L)$. By Whitehead's first Lemma, $I\in B^1(\Lg,\Lg_L)$. Hence there exists an $e\in \Lg$ with 
$R(e)=I$, where $R(x)$ denotes the right multiplication operator. The adjoint operators $\ad (x)=L(x)-R(x)$ have trace 
zero, since $\Lg$ is perfect. So do all $L(x)$ and hence all $R(x)$. Then we obtain $n=\tr (I)=\tr (R(e))=0$, 
a contradiction.
\end{proof}

Helmstetter \cite{HEL} proved more generally that if $\Lg$ is {\em perfect}, i.e., if $\Lg=[\Lg,\Lg]$, then $\Lg$ does
not admit a pre-Lie algebra structure. We have the following canonical bijections (up to suitable equivalence).

\begin{prop}\label{7.9}
There is a canonical bijection between left-invariant affine structures on $G$
and pre-Lie algebra structures on $\Lg$.
\end{prop}

\begin{prop}\label{7.10}
There is a canonical bijection between simply transitive affine actions of $G$ and complete pre-Lie algebra structures 
on $\Lg$.
\end{prop}

Here a pre-Lie algebra structure on $\Lg$ is {\em complete}, if all right multiplications
$R(x)$ in $\End(\Lg)$ are nilpotent. A left-invariant affine structure on $G$ is complete if and only if
the corresponding pre-Lie algebra structure on the Lie algebra of $G$ is complete, see \cite{SEG}.
Hence the algebraic analogue of Milnor's question is as follows.

\begin{qu}[Milnor 1977]\label{7.11}
Does every solvable Lie algebra over a field of characteristic zero admit a (complete)
pre-Lie algebra structure?
\end{qu}

In the nilpotent case the different versions of Milnor's question are equivalent. By Proposition $\ref{5.3}$
we obtain also a correspondence to affine crystallographic actions. The counterexamples to Milnor's question
are given by $n$-dimensional nilpotent Lie algebras not admitting a faithful linear representation of degree $n+1$. 
This is based on the following important observation, see \cite{BEN}.

\begin{prop}
Let $G$ be a $n$-dimensional Lie algebra $\Lg$ over a field $K$ of characteristic zero. Suppose that $\Lg$ admits a
pre-Lie algebra structure. Then $\Lg$ admits a faithful linear Lie algebra representation
$\phi\colon \Lg\ra \mathfrak{gl}_{n+1}(K)$ of degree $n+1$.
\end{prop}

This motivates to  study a refinement of Ado's theorem.

\begin{defi}
Let $\Lg$ be a finite-dimensional Lie algebra over a field $K$ of dimension $n$. Denote by $\mu(\Lg)$
the minimal dimension of a faithful linear representation of $\Lg$.
\end{defi}

By the Ado-Iwasawa theorem, $\mu(\Lg)$ is always finite. However the proofs for Ado's theorem do not give
good upper bounds for $\mu(\Lg)$. The following result was proved in \cite{BU7}.

\begin{thm}
Let $\Lg$ be a $k$-step nilpotent Lie algebra of dimension $n$ over a field of characteristic zero.
Then we have
\[
\mu(\Lg)\le \sum_{j=0}^{k} \binom{n-j}{k-j}p(j)< 3\cdot \frac{2^n}{\sqrt{n}}.
\]
\end{thm}

In the general case we have the following result \cite{BU40}.

\begin{thm}
Let $\Lg$ be a Lie algebra with $r$-dimensional solvable radical and nilradical $\Ln$ over an
algebraically closed field of characteristic zero. Then we have
\[
\mu(\Lg)\le \mu(\Lg/\Ln)+ 3\cdot \frac{2^r}{\sqrt{r}}.
\]
\end{thm}

For the $10$-dimensional counterexamples to Milnor's question we proved that $12\le \mu(\Lg)\le 18$, but we do not know 
the exact value in all cases. \\[0.2cm]
The canonical bijections of Proposition $\ref{7.9}$ and  $\ref{7.10}$ can be generalized to nil-affine transformations
and post-Lie algebra structures, see \cite{BU41}.

\begin{thm}
Let $G$ and $N$ be connected and simply connected nilpotent Lie groups. Then there exists a simply transitive action
by nil-affine transformations of $G$ on $N$ if and only if the corresponding pair
of Lie algebras  $(\Lg,\Ln)$ admits a complete post-Lie algebra structure.
\end{thm}

In the classical case $N=\R^n$ a complete post-Lie algebra structure on $(\Lg,\R^n)$ is just a complete 
{\em pre-Lie algebra structure} on $\Lg$. In the other extreme case $G=\R^n$ a complete post-Lie algebra structure on
$(\R^n,\Ln)$ is a complete {\em LR-structure} on $\Ln$, see \cite{BU38}.\\[0.2cm]
The definition of a post-Lie algebra and a post-Lie algebra structure are as follows \cite{VAL,BU41}.

\begin{defi} 
A {\em post-Lie algebra} $(V,\cdot,\{\, ,\})$ is a vector space $V$ over a field
$K$ equipped with two $K$-bilinear operations $x\cdot y$ and $\{x,y\}$, such that
$\Lg=(V,\{\, ,\})$ is a Lie algebra, and
\begin{align*}
\{x,y\}\cdot z & = (y\cdot x)\cdot z - y\cdot (x\cdot z) -(x\cdot y)\cdot z+x\cdot (y\cdot z) \\
x\cdot \{y,z\} & = \{x\cdot y,z\}+\{y,x\cdot z\} 
\end{align*}
for all $x,y,z \in V$.
\end{defi}

Note that if $\Lg$ is abelian then $(V,\cdot)$ is a pre-Lie algebra. We can associate to a post-Lie algebra 
$(V,\cdot,\{,\})$ a second Lie algebra $\Ln=(V,[\, ,])$ via
\begin{align*}
[x,y] & := x\cdot y-y\cdot x +\{x,y\}. 
\end{align*}
This Lie bracket satisfies the following identity
\begin{align*}
[x,y]\cdot z & = x\cdot (y\cdot z)-y\cdot (x\cdot z), 
\end{align*}
i.e., the post-Lie algebra is a  {\em left module} over the Lie algebra $\Ln$. 

\begin{defi}\label{plt}
Let $(\Lg, [x,y])$, $(\Ln, \{x,y\})$ be two Lie brackets on a vector space $V$.
A {\em post-Lie algebra structure on the pair $(\Lg,\Ln)$} is a $K$-bilinear product $x\cdot y$ satisfying the identities
\begin{align*}
x\cdot y -y\cdot x & = [x,y]-\{x,y\} \\
[x,y]\cdot z & = x\cdot (y\cdot z)-y\cdot (x\cdot z) \\
x\cdot \{y,z\} & = \{x\cdot y,z\}+\{y,x\cdot z\}
\end{align*}
for all $x,y,z \in V$.
\end{defi}

These identities imply the identities given before, so that $(V,\cdot,[\, ,])$ is a post-Lie algebra with associated 
Lie algebra $\Ln$. If $\Ln$ is abelian then the conditions of a post-Lie algebra structure reduce to the conditions
\begin{align*}
[x,y] & = x\cdot y-y\cdot x, \\
[x,y]\cdot z & = x\cdot (y\cdot z)-y\cdot (x\cdot z),
\end{align*}
so that $x\cdot y$ is a pre-Lie algebra structure on $\Lg$. On the other hand, if $\Lg$ is abelian then the 
conditions reduce to
\begin{align*}
x\cdot y-y\cdot x & = -\{x,y\} \\
x\cdot (y\cdot z)& = y\cdot (x\cdot z), \\
(x\cdot y)\cdot z & =(x\cdot z)\cdot y,
\end{align*}
so that $-x\cdot y$ is an LR-structure on $\Ln$. 

\section{Milnor's question for nil-affine transformations}

Milnor's question $\ref{6.6}$ and the algebraic version $\ref{7.11}$ can be asked more generally for
nil-affine transformations and post-Lie algebra structures. So we may ask the
following existence question.

\begin{qu}
Exactly which pairs of Lie algebras $(\Lg,\Ln)$ over a given vector space $V$ over a field of characteristic zero 
admit a post-Lie algebra structure?
\end{qu}

For the correspondence to nil-affine transformations we would need to consider {\em complete} post-Lie algebra
structures, see \cite{BU41}, but we would like to ask more generally for all post-Lie algebra structures.
Of course this question is very ambitious and it is not clear how a complete answer should look like.
It seems reasonable to study here first certain algebraic properties of $\Lg$ and $\Ln$, such as being abelian, nilpotent,
solvable, simple, semisimple, reductive and complete as the most basic ones. \\[0.2cm]
If $\Ln$ is abelian we are back to Milnor's original question and we ask exactly which Lie algebras $\Lg$ admit a 
pre-Lie algebra structure. This is as we already know a difficult question and there are only partial answers. 
For example, if $\Lg$ is semisimple or more generally perfect, then $\Lg$ does not admit any pre-Lie algebra structure, 
see Proposition $\ref{7.8}$ and \cite{HEL}. If $\Lg$ is reductive, the question is already open. Certainly $\Lg\Ll_n(K)$ 
does admit a pre-Lie algebra structure, but on the other hand, there are several restrictions. For example, we have the 
following result, see \cite{BU4}.

\begin{prop}
Let $\Lg=\La\oplus \Ls$ be a reductive Lie algebra, where $\Ls$ is simple and $\La$ is the center of $\Lg$
with $\dim (\La)=1$. Then $\Lg$ admits a pre-Lie algebra structure if and only if $\Ls\cong \Ls\Ll_n(K)$ for some
$n\ge 2$.
\end{prop}

For more results and details concerning the reductive case and \'etale affine representations of reductive groups
see \cite{BAU,BU4,BU53,BU56}. \\[0.2cm]
On the other hand, if $\Lg$ is abelian, then we ask which Lie algebras exactly admit an LR-structure. This question
is more accessible and we have obtained several results, see \cite{BU34}.

\begin{prop}
Let $\Ln$ be a Lie algebra admitting an LR-structure. Then $\Ln$ is two-step solvable.
\end{prop}

However, not every two-step solvable Lie algebra admits an LR-structure.

\begin{prop}
There are $3$-step nilpotent Lie algebras with $4$ generators of dimension $n\ge 13$ not admitting any LR-structure.
\end{prop}

There are no such examples with less than $4$ generators.

\begin{prop}
Let $\Ln$ be a $2$-step nilpotent Lie algebra or a $3$-step nilpotent Lie algebra with at most $3$ generators.
Then $\Lg$ admits a complete LR-structure.
\end{prop}

For further results we refer to \cite{BU34,BU38}. \\[0.2cm]
For the general case concerning post-Lie algebra structures on pairs of Lie algebras $(\Lg,\Ln)$ we also have several 
results, see \cite{BU41,BU44,BU51,BU58,BU59, END}. Let us explain some of them. 

\begin{prop}
Suppose that $(\Lg,\Ln)$ admits a post-Lie algebra structure, where $\Lg$ is nilpotent. Then $\Ln$ is solvable.
If $\Lg$ is nilpotent with $H^0(\Lg,\Ln)=0$, then $\Ln$ is nilpotent.
\end{prop}

In case one of the Lie algebras is semisimple, but the other Lie algebra not, we have the following result.

\begin{prop}
Let $(\Lg,\Ln)$ be a pair of Lie algebras, where $\Lg$ is semisimple and $\Ln$ is solvable. Then $(\Lg,\Ln)$
does not admit a post-Lie algebra structure.
\end{prop}

The situation is not symmetric in $\Lg$ and $\Ln$.

\begin{prop}
Let $(\Lg,\Ln)$ be a pair of Lie algebras, where $\Ln$ is semisimple and $\Lg$ is solvable and unimodular. 
Then $(\Lg,\Ln)$ does not admit a post-Lie algebra structure.
\end{prop}

The unimodularity assumption is essential here. Otherwise any triangular decomposition
of $\Ln$ induces an obvious post-Lie algebra structure on $(\Lg,\Ln)$, where
$\Lg$ is solvable but not unimodular. \\[0.2cm]
In case one of the Lie algebras $\Lg$, $\Ln$ is simple we have the following results.

\begin{prop}
Suppose that $(\Lg,\Ln)$ admits a post-Lie algebra structure, where $\Lg$ is simple. Then $\Ln$ is simple
and isomorphic to $\Lg$. The post-Lie algebra product then is either $x\cdot y=0$ with $[x,y]=\{x,y\}$, or
$x\cdot y=[x,y]$ with $[x,y]=-\{x,y\}$.
\end{prop}

If we interchange the roles of $\Lg$ and $\Ln$ we only can prove the following result, see \cite{BU64}.

\begin{prop}
Suppose that $(\Lg,\Ln)$ admits a post-Lie algebra structure, where $\Ln$ is simple and $\Lg$ is reductive.
Then $\Lg$ is simple and isomorphic to $\Ln$.
\end{prop}

In case both $\Lg$ and $\Ln$ are semisimple, but not simple, we can have many interesting post-Lie algebra structures.

\begin{ex}
Let $\Lg$ and $\Ln$ both isomorphic to $\Ls\Ll_2(\C)\oplus \Ls\Ll_2(\C)$. Then there exist non-trivial post-Lie 
algebra structures on $(\Lg,\Ln)$. If $\Ln=\Ls\Ll_2(\C)\oplus \Ls\Ll_2(\C)$ has the basis $(e_1,f_1,h_1,e_2,f_2,h_2)$
with Lie brackets
$$
\begin{array}{ll}
\{e_1,f_1\} = h_1, & \{e_2,f_2\} = h_2, \\
\{e_1,h_1\} = -2e_1, & \{e_2,h_2\} = -2e_2, \\
\{f_1,h_1\} = 2f_1, & \{f_2,h_2\} = 2f_2, 
\end{array}
$$
then the following product defines a post-Lie algebra structure on $(\Lg,\Ln)$:
$$
\begin{array}{lll}
e_1\cdot e_2 = -4e_2+h_2,  & f_1\cdot e_2 = 2e_2-h_2,  & h_1\cdot e_2 = 6e_2-2h_2, \\
e_1\cdot f_2 = 4f_2+4h_2,  & f_1\cdot f_2 = -2f_2-h_2, & h_1\cdot f_2 = -6f_2-4h_2, \\
e_1\cdot h_2 = -8e_2-2f_2, & f_1\cdot h_2 = 2e_2+2f_2, & h_1\cdot h_2 = 8e_2+4f_2.
\end{array}
$$
\end{ex}

Here the Lie brackets of $\Lg$ are given by
$$
\begin{array}{lll}
[e_1,f_1]=h_1,                   & [f_1,h_1]=2f_1,      & [h_1,f_2]=-6f_2-4h_2, \\
\left[e_1,h_1\right]=-2e_1,      & [f_1,e_2]=2e_2-h_2,  & [h_1,h_2]=8e_2+4f_2, \\
\left[e_1,e_2\right]=-4e_2+h_2,  & [f_1,f_2]=-2f_2-h_2, & [e_2,f_2]=h_2, \\
\left[e_1,f_2\right]=4f_2+4h_2,  & [f_1,h_2]=2e_2+2f_2, & [e_2,h_2]=-2e_2, \\
\left[e_1,h_2\right]=-8e_2-2f_2, & [h_1,e_2]=6e_2-2h_2, &  [f_2,h_2]=2f_2.
\end{array}
$$
It is easy to see that $\Lg$ is isomorphic to $\Ls\Ll_2(\C)\oplus \Ls\Ll_2(\C)$. \\[0.2cm]
The following table shows, what we know about the existence of post-Lie algebra structures on pairs
$(\Lg,\Ln)$, with respect to the seven different classes of Lie algebras given below. So more precisely 
the classes are abelian, nilpotent non-abelian, solvable non-nilpotent, simple, semisimple non-simple, reductive non-semisimple,
non-abelian and complete non-semisimple Lie algebras.
\vspace*{0.5cm}
\begin{center}
\begin{tabular}{l|lllllll}
$(\Lg,\Ln)$ & \color{blue}{$\Ln$ abe} & \color{green}{$\Ln$ nil} & $\Ln$ sol & \color{red}{$\Ln$ sim}
& \color{orange}{$\Ln$ sem} & \color{purple}{$\Ln$ red} &  \color{cyan}{$\Ln$ com} \\[1pt]
\hline
\color{blue}{$\Lg$ abelian} & $\ck$ & $\ck$ & $\ck$ & $-$ & $-$ & $-$ & $\ck$ \\[1pt]
\color{green}{$\Lg$ nilpotent} & $\ck$ & $\ck$ & $\ck$ & $-$ & $-$ & $-$ & $\ck$ \\[1pt]
 $\Lg$ solvable & $\ck$ & $\ck$ & $\ck$ & $\ck$ & $\ck$ & $\ck$ & $\ck$ \\[1pt]
\color{red}{$\Lg$ simple} & $-$ & $-$ & $-$ & $\ck$ & $-$ & $-$ & $-$ \\[1pt]
\color{orange}{$\Lg$ semisimple} & $-$ & $-$ & $-$ & $-$ & $\ck$ & $?$ & $-$ \\[1pt]
\color{purple}{$\Lg$ reductive}  & $\ck$ & $?$ & $?$ & $?$ & $?$ & $\ck$ & $\ck$ \\[1pt]
\color{cyan}{$\Lg$ complete}  & $\ck$ & $\ck$ & $\ck$ & $?$ & $?$ & $\ck$ & $\ck$ \\[1pt]
\end{tabular}
\end{center}
\vspace*{0.5cm}
Note that a checkmark only means that there is {\em some} pair $(\Lg,\Ln)$ with the given algebraic properties
admitting a post-Lie algebra structure. It does not imply that all such pairs admit a post-Lie algebra structure. \\[0.2cm]
Besides existence of post-Lie algebra structures it is also interesting to obtain classification results.
For the general case such results are difficult to obtain. There are only some classifications in low dimensions.
We refer to \cite{BU58} for a classification of post-Lie algebra structures on $(\Lg,\Ln)$, where both $\Lg$ and $\Ln$ 
are isomorphic to the $3$-dimensional Heisenberg Lie algebra. We have much better classification results for
{\em commutative} post-Lie algebra structures, which will be discussed in the next section.

\section{Commutative post-Lie algebra structures}

A post-Lie algebra structure $(V,\cdot)$ on a pair $(\Lg,\Ln)$ is called {\em commutative}, if the
algebra product is commutative, i.e., if $x\cdot y=y\cdot x$ for all $x,y\in V$. This implies that $[x,y]=\{x,y\}$,
so that the Lie algebras $\Lg$ and $\Ln$ are equal. We only write $\Lg$ instead of the pair $(\Lg,\Lg)$.

\begin{defi}
A {\it commutative post-Lie algebra structure}, or {\em CPA-structure} on a Lie algebra $\Lg$
is a $K$-bilinear product $x\cdot y$ satisfying the identities:
\begin{align*}
x\cdot y & =y\cdot x \\
[x,y]\cdot z & = x\cdot (y\cdot z) -y\cdot (x\cdot z) \\
x\cdot [y,z] & = [x\cdot y,z]+[y,x\cdot z] 
\end{align*}
for all $x,y,z \in V$.
\end{defi}

There is always the {\it trivial} CPA-structure on $\Lg$, given by $x\cdot y=0$ for all $x,y\in \Lg$. 
Any CPA-structure on a semisimple Lie algebra over a field of characteristic zero is trivial, see \cite{BU51}.
This was generalized in \cite{BU52} as follows.

\begin{prop}
Any CPA-structure on a perfect Lie algebra of characteristic zero is trivial.
\end{prop}

For complete Lie algebras one can classify all CPA-structures. A Lie algebra $\Lg$ is called {\em complete}, if
$Z(\Lg)=0$ and $\Der(\Lg)=\Inn(\Lg)$. This is equivalent to the cohomological conditions $H^0(\Lg,\Lg)=H^1(\Lg,\Lg)=0$.
A complete Lie algebra is called {\em simply-complete}, if $\Lg$ does not have a non-trivial complete ideal.
Every complete Lie algebra can be written as the direct sum of simply-complete Lie algebras. We have the following
result \cite{BU52}.

\begin{thm}
Let $\Lg$ be a complex simply-complete Lie algebra with nilradical $\Ln$. Suppose that $\Lg$ is not metabelian and
that $\Ln=[\Lg,\Ln]$. Then there is a bijective correspondence between CPA-structures on $\Lg$ and elements
$z\in Z([\Lg,\Lg])$, given by
\[
x\cdot y =[[z,x],y].
\]
\end{thm}

We believe that the condition $\Ln=[\Lg,\Ln]$ is automatically satisfied for complete Lie algebras. However, we could
not find this statement with a proof in the literature. The only simply-complete metabelian Lie algebra is the 
$2$-dimensional non-abelian Lie algebra $\Lr_2(\C)$, where we can classify all CPA-structures directly. \\[0.2cm]
There are also classification results concerning CPA-structures on nilpotent Lie algebras. 
An important fact here is the following, see \cite{BU57}.

\begin{thm}\label{9.4}
Let $\Lg$ be a nilpotent Lie algebra over a field of characteristic zero satisfying $Z(\Lg)\subseteq [\Lg,\Lg]$.
Then every CPA-structure on $\Lg$ is complete, i.e., all left multiplications $L(x)$ are nilpotent.
\end{thm}

In this case we have 
\[
L(Z(\Lg))^{\lceil \frac{\dim Z(\Lg)+1}{2} \rceil}(\Lg)=0.
\]

\begin{defi}
A CPA-structure $(V,\cdot)$ on $\Lg$ is called {\em associative} if $\Lg\cdot [\Lg,\Lg]=0$. It is called
{\em central} if $\Lg\cdot \Lg\subseteq Z(\Lg)$.
\end{defi}

The first part of the definition is justified by the following lemma \cite{BU63}.

\begin{lem}
Let $(V,\cdot)$ be a CPA-structure on a Lie algebra $\Lg$. Then we have  $\Lg\cdot [\Lg,\Lg]=0$ if and only
if the algebra $(V,\cdot)$ is associative.
\end{lem} 

It is easy to see that every central CPA-structure on $\Lg$ is associative and conversely that every associative
CPA-structure on $\Lg$ satisfies $\Lg\cdot \Lg\subseteq Z([\Lg,\Lg])$. Also, every central CPA-structure on $\Lg$ 
satisfies $\Lg\cdot Z(\Lg)=0$. \\[0.2cm]
If $\dim Z(\Lg)=1$ then the formula after Theorem $\ref{9.4}$ yields the following corollary.

\begin{cor}
Let $\Lg$ be a nilpotent Lie algebra over a field of characteristic zero satisfying $Z(\Lg)\subseteq [\Lg,\Lg]$ and
$\dim Z(\Lg)=1$. Then every CPA-structure on $\Lg$ satisfies $\Lg\cdot Z(\Lg)=0$.
\end{cor}

In particular, every CPA-structure on a filiform nilpotent Lie algebra $\Lg$ satisfies $\Lg\cdot Z(\Lg)=0$.
On the other hand, not all CPA-structures on a filiform nilpotent Lie algebra are central or associative.
But we have shown the following result in \cite{BU63}.

\begin{thm}
Let $\Lg$ be a complex filiform Lie algebra of solvability class $d\ge 3$. Then every CPA-structure $(V,\cdot)$ on
$\Lg$ is associative and the algebra $(V,\cdot)$ is Poisson-admissible.
\end{thm}

For certain families of filiform nilpotent Lie algebras a classification of all CPA-structures is possible \cite{BU63,END}.
As an example let us consider the Witt Lie algebra.

\begin{defi}
The Witt Lie algebra $W_n$ for $n\ge 5$ over a field of characteristic zero is defined by the Lie brackets
\begin{align*}
[e_1,e_j] & = e_{j+1}, \quad 2\le j\le n-1,\\[0.1cm]
[e_i,e_j] & = \frac{6(j-i)}{j(j-1)\binom{j+i-2}{i-2}} e_{i+j}, \quad 2\le i\le \frac{n-1}{2},\; i+1\le j\le n-i,
\end{align*}
where $(e_1,\ldots ,e_n)$ is an {\em adapted basis} for $W_n$.
\end{defi}

To give a CPA-structure $(V,\cdot)$ on $\Lg$ explicitly it is enough to list the non-zero products $e_i\cdot e_j$
for all $1\le i\le j\le n$.

\begin{prop}\label{3.15}
Every CPA-structure on the complex Witt algebra $W_n$ for $n\ge 7$ with respect to an adapted basis $(e_1,\ldots ,e_n)$ is given 
as follows,
\begin{align*}
e_1\cdot e_1 & = \al e_{n-2}+\be e_{n-1}+ \ga e_n,\\
e_1\cdot e_2 & = \frac{6(n-4)}{(n-2)(n-3)} \al e_{n-1}+\de e_n,\\
e_2\cdot e_2 & = \ep e_n,
\end{align*}
where $\al,\be,\ga,\de,\ep \in \C$ are arbitrary parameters.
\end{prop}

Note that all CPA-structures on the Witt algebra are associative but not necessarily central. \\[0.2cm]
We also have a result concerning CPA-structures on free-nilpotent Lie algebras $F_{g,c}$ with $g\ge 2$ generators and
nilpotency class $c\ge 2$, see \cite{BU57}.

\begin{thm}
All CPA-structures on $F_{3,c}$ with $c\ge 3$ are central.
\end{thm}

The result is not true for $F_{3,2}$. We believe that all CPA-structures on $F_{g,c}$ with $g\ge 2$ and $c\ge 3$ 
are central. However, we could only prove a part of it so far, see \cite{BU57}. \\[0.2cm]
Finally we have determined the CPA-structures on certain infinite-dimensional Lie algebras, e.g., on 
Kac-Moody algebras \cite{BU60}. For the infinite-dimensional Witt algebra ${\mathcal W}$ in characteristic zero with a set of 
basis vectors $\{e_i\}$ and Lie brackets
\[
[e_i,e_j]=(j-i)e_{i+j}
\]
we have that all CPA-structures on ${\mathcal W}$ are trivial. Note that in case the basis is finite, 
$\mathcal{W}$ is isomorphic to $W_n$ for some $n$.

\section*{Acknowledgments}

Dietrich Burde is supported by the Austrian Science Foun\-da\-tion FWF, grant I3248.


\begin{thebibliography}{99}

\bibitem{AMS} H. Abels, G. A. Margulis, G. A. Soifer: {\it Properly
discontinuous groups of affine transformations with orthogonal
linear part}. C.\ R.\ Acad.\ Sci.\ Paris \textbf{324} (1997), 253--258.

\bibitem{AMS2} H. Abels, G. A. Margulis, G. A. Soifer: {\em The Auslander conjecture for groups leaving a form of 
signature $(n−2,2)$ invariant}. 
Israel J.\ Math.\ \textbf{148} (2005), 11--21. 

\bibitem{AUS1} L. Auslander: {\it The structure of complete locally
affine manifolds}. Topology \textbf{3} (1964), 131--139.

\bibitem{AUS2} L. Auslander: {\em An Account of the Theory of Crystallographic Groups}.
Proceedings of the American Mathematical Society, Vol. \textbf{16}, Issue 6 (1965), 1230--1236.

\bibitem{AUS3} L. Auslander: {\em Simply transitive groups of affine motions}.
Amer.\ J.\ Math.\ \textbf{99} (1977), no. 4, 809--826. 

\bibitem{BAU} O. Baues: {\em Left-symmetric Algebras for $\Lg\Ll_n$},
Trans.\ Amer.\ Math.\ Soc.\ {\bf 351} (1999), no. 7, 2979--2996.

\bibitem{BKK} G. Benkart, A. Kostrikin, M. Kuznetsov: {\em Finite-dimensional simple Lie algebras with a 
nonsingular derivation}. 
J.\ Algebra \textbf{171} (1995), no. 3, 894--916. 

\bibitem{BEZ} J. P. Benz\'ecri: {\it Vari\'et\'es  localement affines}.
Th\`ese, Princeton Univ.\ , Princeton, N.\ J.\ (1955).

\bibitem{BI1} L. Bieberbach: {\em \"Uber die Bewegungsgruppen der Euklidischen R\"aume}. (German) 
Math.\ Ann.\ \textbf{70} (1911), no. 3, 297--336.

\bibitem{BI2} L. Bieberbach: {\em \"uber die Bewegungsgruppen der Euklidischen Räume. Die Gruppen mit einem endlichen 
Fundamentalbereich}. (German) 
Math.\ Ann.\ \textbf{72} (1912), no. 3, 400--412. 

\bibitem{BEN} Y. Benoist: {\it Une nilvari\'et\'e non affine}.
J.\ Diff.\ Geom.\ \textbf{41} (1995), 21--52.

\bibitem{BBNWZ} H. Brown, R. B\"ulow, J. Neub\"user, H. Wondratschek, H. Zassenhaus: {\em Crystallographic groups of 
four-dimensional space}. Wiley Monographs in Crystallography, Wiley-Interscience, New York-Chichester-Brisbane, 
1978. xiv+443 pp.  

\bibitem{BU2} D. Burde: {\it Left-symmetric structures on simple modular Lie algebras}.
Journal of Algebra, Vol. \textbf{169} (1994), Issue 1, 112--138.

\bibitem{BU3} D. Burde, F. Grunewald: {\it Modules for certain Lie algebras of maximal class}.
Journal of Pure and Applied Algebra, Vol. \textbf{99} (1995), Issue 3, 239--254.

\bibitem{BU4} D. Burde: {\it Left-invariant affine structures on reductive Lie groups}.
Journal of Algebra, Vol. \textbf{181} (1996), Issue 3, 884--902.

\bibitem{BU5} D. Burde: {\it Affine structures on nilmanifolds}. International Journal of
Mathematics, Vol. \textbf{7} (1996), Issue 5, 599--616.

\bibitem{BU7} D. Burde: {\it A refinement of Ado's Theorem}. Archiv der Mathematik, Vol.
\textbf{70} (1998), Issue 2, 118--127.

\bibitem{BU19} D. Burde, K. Dekimpe, S. Deschamps: {\it The Auslander conjecture for NIL-affine
crystallographic groups}. Mathematische Annalen, Vol. \textbf{332} (2005), Issue 1, 161--176.

\bibitem{BU24} D. Burde: {\it Left-symmetric algebras, or pre-Lie algebras in geometry
and physics}. Central European Journal of Mathematics, Vol. \textbf{4} (2006), Issue 3, 323--357.

\bibitem{BU34} D. Burde, K. Dekimpe and S. Deschamps: {\it LR-algebras}.
Contemporary Mathematics, Vol. \textbf{491} (2009), 125--140.

\bibitem{BU35} D. Burde, B. Eick, W. A. de Graaf: {\it Computing faithful representations
for nilpotent Lie algebras}. Journal of Algebra, Vol. \textbf{22} (2009), Issue 3, 602--612.

\bibitem{BU38} D. Burde, K. Dekimpe, K. Vercammen: {\it Complete LR-structures on solvable
Lie algebras}. Journal of Group Theory, Vol. \textbf{13} (2010), Issue 5, 703--719.

\bibitem{BU40} D. Burde, W. A. Moens: {\it Faithful Lie algebra modules and quotients of
the universal enveloping algebra}.
Journal of Algebra, Vol. \textbf{325} (2011), Issue 1, 440--460.

\bibitem{BU41} D. Burde, K. Dekimpe and K. Vercammen: {\it Affine actions on Lie groups
and post-Lie algebra structures}.
Linear Algebra and its Applications, Vol. \textbf{437} (2012), Issue 5, 1250--1263.

\bibitem{BU44} D. Burde, K. Dekimpe: {\it Post-Lie algebra structures and generalized
derivations of semisimple Lie algebras}.
Moscow Mathematical Journal, Vol. \textbf{13} (2013), Issue 1, 1--18.

\bibitem{BU51} D. Burde, K. Dekimpe: {\it Post-Lie algebra structures on pairs of Lie algebras}.
Journal of Algebra, Vol. \textbf{464} (2016), 226--245.

\bibitem{BU52} D. Burde, W. A. Moens: {\it Commutative post-Lie algebra structures on Lie algebras}.
Journal of Algebra, Vol. \textbf{467} (2016), 183--201.

\bibitem{BU53} D. Burde, W. Globke: {\it \'Etale representations for reductive algebraic groups 
with one-dimensional center}.
Journal of Algebra, Vol. \textbf{487} (2017), 200--216.

\bibitem{BU56} D. Burde, W. Globke, A. Minchenko: {\it \'Etale representations for reductive
algebraic groups with factors $Sp_n$ or $SO_n$}.
Transformation Groups, Vol. \textbf{24}, Issue 3, 769-780 (2019).

\bibitem{BU57} D. Burde, K. Dekimpe, W. A. Moens: {\it Commutative post-Lie algebra structures
and linear equations for nilpotent Lie algebras}.
Journal of Algebra, Vol. \textbf{526} (2019), 12--29.

\bibitem{BU58} D. Burde, C. Ender, W. A. Moens: {\it Post-Lie algebra structures for
nilpotent Lie algebras}. International Journal of Algebra and Computation, Vol.
\textbf{28}, Issue 5 (2018), 915--933.

\bibitem{BU59} D. Burde, V. Gubarev: {\em Rota-Baxter operators and post-Lie algebra
structures on semisimple Lie algebras}.
Communications in Algebra, Vol. \textbf{47}, Issue 5, 2280--2296 (2019).

\bibitem{BU60} D. Burde, P. Zusmanovich: {\em Commutative post-Lie algebra structures on
Kac-Moody algebras}.
To appear in Communications of Algebra (2019).

\bibitem{BU63} D. Burde, C. Ender: {\em Commutative Post-Lie algebra structures on nilpotent
Lie algebras and Poisson algebras}. To appear in Linear Algebra and its Applications (2019).

\bibitem{BU64} D. Burde, V. Gubarev: {\em Decompositions of algebras and post-associative
algebra structures}. ArXiv: 1906.09854 (2019).

\bibitem{BUS} P. Buser: {\em A geometric proof of Bieberbach's theorems on crystallographic groups}. 
Enseign.\ Math.\ \textbf{31} (1985), no. 1-2, 137--145.

\bibitem{DEI} K. Dekimpe, P. Igodt: {\em Polycyclic-by-finite groups admit a bounded-degree polynomial structure}. 
Invent.\ Math.\ \textbf{129} (1997), no. 1, 121--140. 

\bibitem{DEK} K. Dekimpe: {\em Any virtually polycyclic group admits a NIL-affine crystallographic action}. 
Topology \textbf{42} (2003), no. 4, 821--832. 

\bibitem{END} C. Ender: {\em Post-Lie algebra structures on classes of Lie algebras}. PhD-thesis,
University of Vienna 2019.

\bibitem{FG} D. Fried, W. Goldman: {\em Three-dimensional affine crystallographic groups}. 
Adv.\ Math.\ \textbf{47} (1983), 1--49.

\bibitem{GOL} W. M. Goldman: {\em  Two papers which changed my life: Milnor's seminal work on flat manifolds and bundles}. 
Frontiers in complex dynamics, Princeton Math.\ Ser.\ \textbf{51} (2014),  679--703.

\bibitem{GRP} M. Gromov, I. Piatetski-Shapiro: {\em Non-arithmetic groups in Lobachevsky spaces}.
Inst.\ Hautes \'Etudes Sci.\ Publ.\ Math.\ No. \textbf{66} (1988), 93--103. 

\bibitem{HEL} J. Helmstetter: {\it Radical d'une alg\`ebre sym\'etrique a gauche}. 
Ann.\ Inst.\ Fourier \textbf{29} (1979), 17--35.

\bibitem{HIL} D. Hilbert: {\em Mathematical problems}.
Bulletin of the AMS \textbf{87} (1902), no. 10, 437--479.

\bibitem{JAC} N. Jacobson: {\em A note on automorphisms and derivations of Lie algebras}. 
Proc.\ Am.\ Math.\ Soc.\ \textbf{6} (1955), 281--283.

\bibitem{MAR} G. A. Margulis: {\em Complete affine locally flat manifolds
with a free fundamental group}. J.\ Soviet Math.\ \textbf{36} (1987), 129--139.

\bibitem{MIL} J. Milnor: {\it On fundamental groups of complete affinely flat manifolds}. 
Advances in Math.\ \textbf{25} (1977), 178--187.

\bibitem{PLS} W. Plesken, T. Schulz: {\em Counting crystallographic groups in low dimensions}. 
Experiment.\ Math.\ \textbf{9} (2000), no. 3, 407--411.

\bibitem{SCH}  R. L. E. Schwarzenberger: {\em Crystallography in spaces of arbitrary dimension}. 
Proc.\ Cambridge Philos.\ Soc.\ \textbf{76} (1974), 23--32.

\bibitem{SEG} D. Segal: {\it The structure of complete left-symmetric algebras}.
Math.\ Ann.\ \textbf{293} (1992), 569--578.

\bibitem{THU} W. P. Thurston: {\em Three-dimensional geometry and topology}. Vol. 1. 
Edited by Silvio Levy. Princeton Mathematical Series, \textbf{35}. Princeton University Press, Princeton, NJ, 1997. x+311 pp.

\bibitem{TOM} G. Tomanov: {\em Properly discontinuous group actions on affine homogeneous spaces}. (English summary)
Reprinted in Proc.\ Steklov Inst.\ Math. \textbf{292} (2016), no. 1, 260--271.
Tr.\ Mat.\ Inst.\ Steklova \textbf{292} (2016), Algebra, Geometriya i Teoriya Chisel, 268--279.

\bibitem{VAL} B. Vallette: {\it Homology of generalized partition posets}.
J.\ Pure and Applied Algebra \textbf{208}, No. 2 (2007), 699--725.

\bibitem{ZAS} H. Zassenhaus: {\em \"Uber einen Algorithmus zur Bestimmung der Raumgruppen}. (German) 
Comment.\ Math.\ Helv.\ \textbf{21}, (1948). 117--141.


\end{thebibliography}
\end{document}